\documentclass[11pt]{article}

\usepackage{amsmath, amssymb, amsthm, verbatim,enumerate,bbm,color} 
\usepackage{indentfirst}
\usepackage{tikz}
\usetikzlibrary{arrows}

\title{Optimal Threshold for a Random Graph to be 2-Universal}

\author{Asaf Ferber \thanks{Department of Mathematics, MIT. Email: ferbera@mit.edu}
	\and Gal Kronenberg
		\thanks{School of Mathematical Sciences, Raymond and Beverly Sackler Faculty of Exact Sciences, Tel Aviv University, Tel Aviv, 6997801, Israel. Email: galkrone@mail.tau.ac.il. Supported
			by the Prof. Rahamimoff Travel Grants Program for Young Scientists of the US-Israel Binational Science Foundation (BSF)}
	\and Kyle Luh \thanks{Department of Mathematics, Yale University. Email: kyle.luh@yale.edu}}

\date{}
\allowdisplaybreaks

\theoremstyle{plain}
\newtheorem{theorem}{Theorem}[section]
\newtheorem{lemma}[theorem]{Lemma}

\newtheorem{corollary}[theorem]{Corollary}
\newtheorem{conjecture}[theorem]{Conjecture}

\newtheorem{remark}[theorem]{Remark}
\newtheorem{definition}[theorem]{Definition}

\newcommand{\eps}{\varepsilon}

\RequirePackage[normalem]{ulem} 
\RequirePackage{color}\definecolor{RED}{rgb}{1,0,0}\definecolor{BLUE}{rgb}{0,0,1} 

\begin{document}
\maketitle

\begin{abstract}
  For a family of graphs $\mathcal{F}$, a graph $G$ is $\mathcal{F}$-universal if $G$ contains every graph in $\mathcal{F}$ as a (not necessarily induced) subgraph. For the family of all graphs on $n$ vertices and of maximum degree at most two, $\mathcal{H}(n,2)$, we prove that there exists a constant $C$ such that for
$
p \geq C \left( \frac{\log n}{n^2} \right)^{\frac{1}{3}},
$
the binomial random graph $G(n,p)$ is typically $\mathcal{H}(n,2)$-universal.  This bound is optimal up to the constant factor as illustrated in the seminal work of Johansson, Kahn, and Vu for triangle factors. Our result improves significantly on the previous best bound of
$
p \geq C \left(\frac{\log n}{n}\right)^{\frac{1}{2}}
$
due to Kim and Lee. In fact, we prove the stronger result that for the family of all graphs on $n$ vertices, of maximum degree at most two and of girth at least $\ell$, $\mathcal{H}^{\ell}(n,2)$, $G(n,p)$ is typically $\mathcal H^{\ell}(n,2)$-universal when
$
p \geq C \left(\frac{\log n}{n^{\ell -1}}\right)^{\frac{1}{\ell}}.
$
This result is also optimal up to the constant factor. Our results verify (in a weak form) a classical conjecture of Kahn and Kalai.
\end{abstract}

\section{Introduction}
Ever since its introduction by Erd{\H
o}s and R{\'e}nyi in 1960 ~\cite{erdos1960evolution}, the  random graph model has been one of the main objects of
study in probabilistic combinatorics.
Given a positive integer $n$ and a real number $p \in [0,1]$, the random variable $G(n,p)$, referred to as the \emph{binomial random graph}, takes values in the set of labeled graphs with vertex set $[n]$.  The distribution of $G(n,p)$ is such that each pair of elements in $[n]$ forms an edge independently with probability $p$.

Most of the questions considered in this model
have the following generic form: given some \emph{monotone} graph property ${\cal
P}$ (that is, a property that cannot be violated by the addition of extra edges), determine whether $G(n,p)$
satisfies ${\cal P}$ {\em with high probability} (whp), i.e., if the
probability that $G(n,p)$ satisfies ${\cal P}$ tends to $1$ as $n$
tends to infinity. Recall that for a monotone property $\mathcal P_n$ (e.g. ``$G(n,p)$ contains a fixed graph $H$"), a function $q(n)$ is a threshold function for $\mathcal P_n$ if and only if

\[ \Pr\left[G(n,p) \text{ satisfies } \mathcal P_n\right]\rightarrow \left\{ \begin{array}{cc}
       1&  \text{if } p(n)/q(n)=\omega(1)   \\
       0&  \text{ if } p(n)/q(n)=o(1).\end{array} \right. \]

As opposed to graphs $H$ of a fixed size (see e.g \cite{bollobas1998random}, pages 257-274), finding threshold functions for \emph{large} graphs (i.e., of size depending on $n$) is a harder task. In particular, the problem becomes even harder whenever $H$ is a \emph{spanning} structure (i.e., graphs on $n$ vertices).

Erd\H{o}s and R\'{e}nyi's work \cite{erdHos1966existence} on the precise threshold for the appearance of a perfect matching and P\'osa's breakthrough result \cite{posa1976hamiltonian} on Hamilton Cycles are two classical theorems on spanning structures.  As researchers investigated more complex structures (e.g. spanning trees, planar graphs of bounded degree, spanning graphs of bounded degree, etc.) techniques took longer to develop and upper bounds on the corresponding thresholds were more difficult to achieve. As a relaxation, the notion of an \emph{almost-spanning} subgraph, meaning a subgraph on $(1-\eps) n$ vertices for a fixed $\eps >0$, was often imposed. Apparently, almost spanning structures are easier to handle and to achieve spanning versions is often a significant technical hurdle. As a first example we consider the problem of finding an $H$-\emph{factor} in a typical $G(n,p)$(that is, vertex disjoint copies of a given graph $H$, covering all the vertices of $G$). For simplicity, let us assume that $H=K_{\Delta+1}$ (a complete graph on $\Delta+1$ vertices) for some $\Delta\in \mathbb{N}$. In this case, the problem of finding vertex disjoint copies of $H$ covering $(1-\varepsilon)n$ vertices is quite easy and can be readily proved by Janson's inequality (see e.g. \cite{alon2004probabilistic}) provided that $p\geq Cn^{-2/(\Delta+1)}$ (this bound is tight up to the constant factor). Regarding the spanning analog, it took almost three decades (since the problem was posed) until Johansson, Kahn and Vu, in their seminal paper \cite{johansson2008factors}, managed to settle it completely by showing that $p=n^{-2/(\Delta+1)}(\log n)^{1/\binom{\Delta+1}{2}}$ is a threshold function for this property.

Another interesting example is when our target graph $H$ is a tree of maximum degree at most $\Delta=O(1)$. Alon, Krivelevich, and Sudakov in \cite{alon2007embedding}, and independently Balogh, Csaba, Pei, and Samotij in \cite{balogh2010large} showed that one can embed every almost-spanning tree with maximum degree at most $\Delta$ in the random graph $G(n,c/n)$ for a large enough $c$. Clearly, $p$ is optimal up to the constant factor, for if $p=o(1/n)$ it is well known (see e.g. \cite{alon2004probabilistic}) that a typical $G(n,p)$ has no connected components of linear size. For the spanning version, improving upon a bound of $n^{-1+\eps}$ obtained by Krivelevich \cite{krivelevich2010embedding}, Montgomery \cite{montgomery2014embedding} recently managed to prove that a graph $G(n, \Delta \log^5 n/n)$ typically contains a copy of any spanning tree $T$ of maximum degree $\Delta$.

A natural, (and arguably) more complex family of graphs to examine is the family $\mathcal H(n,\Delta)$ consisting of all graphs on $n$ vertices and of maximum degree at most $\Delta$. Observe that a $K_{\Delta+1}$-factor belongs to this class, and therefore, a general bound for graphs $H\in \mathcal H(n,\Delta)$ should be at least $p=n^{-2/(\Delta+1)}(\log n)^{1/\binom{\Delta+1}{2}}$, as demonstrated in \cite{johansson2008factors}. It is indeed conjectured (see e.g, \cite{FLN}) that this bound is enough for all $H\in\mathcal H(n,\Delta)$.

Perhaps the first general result regarding the family $\mathcal H(n,\Delta)$ is due to Alon and F\"uredi \cite{alon1992spanning}. Introducing an embedding technique based on matchings, they showed that for a given
graph $H\in \mathcal H(n,\Delta)$ a
typical $G(n,p)$ contains a copy of $H$, provided that
$p=\omega\left(n^{-1/\Delta}(\log n)^{1/\Delta}\right)$. Note that the bound they obtained on $p$ is quite natural, as in this range it is easy to see that typically every subset of size $\Delta$ has common neighbors in $G(n,p)$ and therefore, at least intuitively, one would expect to find a ``vertex-by-vertex" type embedding.
Surprisingly, using only the second moment method, Riordan \cite{riordan2000spanning}, for a given $H\in \mathcal H(n,\Delta)$, managed to obtain a bound of
$p \geq n^{\frac{-2(\Delta-1)}{\Delta (\Delta+1)}}$, which is within $n^{\Theta(1/\Delta^2)}$ of the conjectured threshold. Recently, by considering $\mathcal H((1-\varepsilon)n,\Delta)$ (i.e. the almost spanning version), the first and third authors together with Nguyen \cite{FLN} were able to show the ``optimal" threshold of $p \geq (n^{-1} \log^{1/\Delta} n)^{2/(\Delta+1)}$ (we say ``optimal" since in the almost spanning case it is believed that the $\log$ is redundant).

In a related (and harder) problem, we consider a family $\mathcal{F}$ of graphs rather than a single graph.  For a family $\mathcal{F}$, a graph $G$ is $\mathcal{F}$-universal if $G$ contains every graph in $\mathcal{F}$ as a subgraph.  This problem has been investigated for a variety of families $\mathcal{F}$: trees \cite{chung1978graphs, chang1976graphs}, spanning trees \cite{bhatt1989universal,chung1979universal,chung1983universal,friedman1987expanding}, planar graphs of bounded degree \cite{bhatt1989universal}, graphs of bounded size \cite{babai1982graphs,rodl1981note},  graphs of bounded degree \cite{alon2007sparse,alon2008optimal,alon2000universality,alon2001near,capalbo1999small}, and spanning graphs of bounded degree \cite{alon2002sparse,johannsen2013expanders}.
We let $\mathcal{H}(n,\Delta)$ denote the family of graphs on $n$ vertices with maximum degree bounded by $\Delta$ and abbreviate $\mathcal{H}(n,\Delta)$-universality as $\Delta$-universality.  The first natural candidate to investigate is the $\Delta = 2$ case.
Recently, in \cite{kim2014universality}, Kim and Lee were able to show that for
$
p \geq C \left(\frac{\log n}{n}\right)^{\frac{1}{2}}
$
$G(n,p)$ is $2$-universal with high probability. In the present paper, we obtain the optimal result for the $\Delta=2$ case, improving the result of Kim and Lee and yielding the first tight result for the universality problem in the \emph{spanning} case.

\begin{theorem}
There exists a constant $C$ such that $G(n,p)$ is $2$-universal with high probability, provided that
$
p \geq C \left( \frac{\log n}{n^2} \right)^{\frac{1}{3}}
$.
\end{theorem}
\begin{remark}
Observe that without loss of generality, by adding edges if necessary, we can assume that each $H\in \mathcal{H}(n,2)$ is a collection of vertex-disjoint cycles (from here and throughout the paper we consider isolated vertices and isolated edges as cycles as well).
\end{remark}
In fact, we show the following stronger result. For a positive integer $\ell\in \mathbb{N}$ we let $\mathcal H^{\ell}(n,2)$ to denote the collection of all graphs $H$ on $n$ vertices for which: $(i)$ $H$ has maximum degree at most $2$, and $(ii)$ every cycle in $H$ is either of length at most $2$ or at least $\ell$. Our main result can now be stated:
\begin{theorem} \label{thm:main}
Let $\ell\geq 3$ be any integer. Then there exists a constant $C(\ell)$ such that $G(n,p)$  is $\mathcal{H}^\ell(n,2)$-universal with high probability, provided that $
p \geq C \left(\frac{\log n}{n^{\ell -1}}\right)^{\frac{1}{\ell}}
$.
\end{theorem}

Our proof of the above theorem utilizes the result of Johansson, Kahn, and Vu \cite{johansson2008factors} in the context of universality.  This surprising application is normally foiled by the need to take a massive union bound.  We avoid this obstruction by combining a union bound over a polynomial number of elements with an elegant argument of Montgomery \cite{montgomery2014embedding}, which exploits the expansion properties of the random graph.  The strategy is to embed a polynomial number of structures at the beginning and then use pseudo-random properties of the graph to generate paths that will extend the embedded structures into embeddings of all the graphs in $\mathcal{H}^\ell(n,2)$.

There has been work done on the general family $\mathcal{H}(n,\Delta)$.
Alon, Capalbo, Kohayakawa, R\H{o}dl, Ruci\'nski and Szemer\'edi \cite{alon2000universality} showed that for the \emph{almost-spanning} graphs of bounded degree, there exists a $c>0$ so that $G(n,p)$ is $\mathcal{H}((1-\eps)n, \Delta)$-universal with high probability when $p \geq c (\log n/n)^{1/\Delta}$.
In \cite{dellamonica2012universality}, Dellamonica, Kohayakawa, R\H{o}dl and Ruc\'{i}nski addressed the \emph{spanning} case and demonstrated that for every $\Delta \geq 2$ there exists a constant $C(\Delta)$ such that if
$
p \geq C \left(\frac{\log^2 n}{n}\right)^{\frac{1}{2\Delta}}
$
then the random graph $G(n,p)$ is $\Delta$-universal with high probability.
In \cite{dellamonica2015improved}, the above bound was improved to
$
p \geq C \left( \frac{\log n}{n}\right)^{\frac{1}{\Delta}}
$
for $\Delta \geq 3$.  Yet still, these bounds fall short of the thresholds that we conjecture below.

An idea that often informs the search for threshold phenomenon is the Kahn-Kalai conjecture.  They define what is known as the \emph{expected threshold}\cite{kahn2007thresholds}, which for a monotone property $\mathcal P_n$ (for example, ``$G(n,p)$ contains a fixed graph $H$"), is the least $p(n)$ such that for each $H'\subseteq H$, the expected number of copies of $H'$ in $G(n,p)$ is at least one.  Following \cite{kahn2007thresholds}, note that it also makes sense if, instead of a fixed $H$,
we consider a sequence $\{H_n\}_n$, of graphs with $|V (H_n)| = n$. Formally,
for an arbitrary $H$, Kahn and Kalai defined $p_E(H)$ to be the least $p$ such that, for every spanning $H'\subseteq H$, $(|V (H')|!/|Aut(H')|)p^{|E(H')}|\geq 1$. Then, even for a fixed graph $H$, the expected threshold is the same as $p_E(H_n)$ if we consider $H_n$ as $H$ plus $n-|V(H)|$ many isolated vertices.

For the class of large graphs (that is, whenever the number of vertices of $H$ is allowed to grow with $n$), one cannot expect $p_E(H)$ to capture the true threshold behavior. For example, suppose that $H_n$ are Hamilton cycles (that is, a cycle on $n$ vertices). In this case, on one hand, the expected number of Hamilton cycles in $G(n,p)$ is clearly $\mu_n=\frac{(n-1)!}{2}p^n$, and therefore the expectation threshold is of order $1/n$. On the other hand, it is well known \cite{posa1976hamiltonian}, that a threshold function for $H_n$ is of order $\log n/n$ (in fact, more precise results are known. The interested reader is referred to \cite{bollobas1998random} and its relevant references). A similar phenomena holds for perfect matchings (that is, a collection of $\lfloor n/2\rfloor$ pairwise disjoint edges), as was proven in \cite{erdos1960evolution}.

These examples lead Kahn and Kalai to the beautiful conjecture that in fact, one cannot lose more than a $\Theta(\log |V(H)|)$ factor in $p_E(H_n)$. Specifically, they conjectured the following (Conjecture 2. in \cite{kahn2007thresholds}):

\begin{conjecture}\label{kahn-kalai} For any $\{H_n\}_n$, a threshold function for the property ``$G(n,p)$ contains $H_n$" is $O(p_E(H_n)\log |V (H_n)|)$.
\end{conjecture}

As the disjoint union of cliques of size $\Delta+1$ are locally the densest of the graphs in $\mathcal{H}(n,\Delta)$, the threshold for the appearance of these $K_{\Delta+1}$ factors is believed to be the worst-case threshold.  Conjecture \ref{kahn-kalai} yields an estimate for the threshold which was later verified by Johansson, Kahn, and Vu \cite{johansson2008factors}.
Theorem \ref{thm:main} suggests that a stronger statement is true in our context, which we pose as a conjecture.
\begin{conjecture}
For $\Delta \geq 3$, a threshold function for `` $G(n,p)$ is $\mathcal{H}(n,\Delta)$-universal" is $p \sim (n^{-1}\log^{1/\Delta} n)^{\frac{2}{\Delta+1}}$.  As this is the threshold for the appearance of $K_{\Delta+1}$ factors, in words, we are saying that the threshold for the appearance of a \emph{single} $K_{\Delta+1}$ factor is also the threshold for the appearance of all $H \in \mathcal{H}(n,\Delta)$ \emph{simultaneously}.
\end{conjecture}

\subsection{Notation} \label{section:notation}
When considering $\mathcal{H}^\ell (n,2)$, we define the convenient constants
\begin{equation} \label{eq:notation}
\eps(\ell) := \frac{1}{3\ell} \text{ , } k:= \frac{12}{\eps} \text{ , and } K(\ell) := 2^{2^{1/\eps}}
\end{equation}

and \emph{with high probability}, abbriviated as \emph{whp} will mean with probability tending to $1$ as $n$ tends to infinity.

The graph theoretic notation is standard and follows that of \cite{west2001introduction}.  We will make use of both directed and undirected graphs.  We let $D(n,p)$ denote a random variable that takes on values in the set of labeled, directed graphs on $[n]$.  $D(n,p)$ is distributed such that every ordered tuple $(u,v) \in [n]^2$ is a directed edge from $u$ to $v$ independently with probability $p$.  For a digraph $D$, we let $V(D)$ and $E(D)$ denote the set of vertices and the set of directed edges, respectively.  Additionally, let $v(D) := |V(D)|$ and $e(D) := |E(D)|$. For a subset $S \subset V(D)$ we use $V[S]$ to denote the induced subgraph by $S$ and for $X\subseteq V(D)$ we let
$$
N^+_G(X) := \{y \in V(G) \setminus X: \exists x \in X \text{ s.t. } \{x,y\} \in E(G)\}.
$$
For two (not necessarily disjoint) subsets $X, Y \subset V(D)$, we set
$$
E_D(X,Y) := \{(x,y) \in E(D): x \in X \text{ and } y \in Y\}
$$
and $e_D(X,Y) := |E_D(X,Y)|$.
Furthermore, let
$$
N^+_D(X,Y) := N^+_D(X) \cap Y
$$
be the \emph{out-neighbors} of $X$ in $Y$.  The notation for undirected graphs is similar and should be clear from context.  Many ceiling and floor functions will be omitted when such discrepancies can be tolerated.

\subsection{Organization of the Paper}
In Section \ref{section:tools} we catalog some well-known probabilistic tools and we introduce technical notation and lemmas specific to our setting.  In particular, the two key lemmas in Section \ref{subsection:connectinglemmas} are modifications of the ``Connecting Lemmas" of Montgomery \cite{montgomery2014embedding}.  As the proofs of these lemmas are analogous to those of \cite{montgomery2014embedding}, we defer the details of the proofs to Section \ref{section:proofconnecting}.  The simple proof of Theorem \ref{thm:main} is featured in Section \ref{section:mainproof}.  In Section \ref{section:preparatorylemmas}, we give an account of the lemmas necessary in the proof of the Connecting Lemmas.  Finally, the delayed proofs of the Connecting Lemmas in Section \ref{section:proofconnecting} occupy the remainder of the paper.

\section{Tools and auxiliary results} \label{section:tools}

In this section we introduce some tools and auxiliary lemmas to be used in our proof of Theorem \ref{thm:main}.
\subsection{Probabilistic Tools}
%
First, we present a result of Bednarska and \L uczak \cite{bednarska2000biased} which is obtained by applying Janson's inequality (see e.g. \cite{janson1988exponential}) to cycles. Before stating it we need the following definition.
\begin{definition}
For a graph $G$ on at least three vertices, we define
$$
m_1(G) := \max_{H \subset G, v(H) \geq 3} \frac{e(H)-1}{v(H)-2}
$$
\end{definition}
\begin{lemma}\cite[Lemma 3]{bednarska2000biased}\label{lemma:cycles}
For every graph $H$ containing a cycle, there exists a constant $c_1 >0$ such that for sufficiently large $n$ and $n^{-1/m_1(G)} \leq p \leq 3 n^{-1/m_1(G)}$ we have
$$
\mathbb{P}(H \nsubseteq G(n,p)) \leq \exp(-c_1 n^2 p)
$$
\end{lemma}

As an almost immediate corollary, we obtain the following.

\begin{corollary} \label{lemma:embeddingcycles}
Let $\alpha >0$ be a constant, and $\ell,K \geq 3$ be integer constants.  There exists a constant $C$ such that for $H$ a cycle of length at least $\ell$ and at most $K$, we can embed $H$ in $G(\alpha n, p)$ with probability $1-\exp(-\omega(n))$, provided that
$
p \geq C \left(\frac{\log n}{n^{\ell -1}} \right)^{1/\ell}
$.
\end{corollary}

\begin{proof}
Lemma \ref{lemma:cycles} yields the result for cycles of length $\ell+1$ and above.  The remaining case of the cycle of length $\ell$ is a standard application of Janson's inequality and is left to the reader.
\end{proof}

\subsection{Pseudo-random properties}

As we are aiming to prove a universality result, we wish to work in a \emph{pseudo-random} graph model rather than a random one. In the universality setting, we lose the luxury of exposing a new random graph whimsically, as we then must repeat these exposures for the large number of graphs in the family.  Here we give a formal description of the pseudo-random properties we are about to use. For convenience, we adopt our notation from \cite{johannsen2013expanders} so we can also make use of some technical (and standard) lemmas from there.

\begin{definition}
Let $n \in \mathbb{N}$, $d \in \mathbb{R}^+$, and let
$$
m:=m(n,d)= \left \lceil \frac{n}{2d} \right \rceil
$$
  A digraph $D$ is an $(n,d)$-expander if $|V(D)|=n$ and $D$ satisfies the following two conditions:
\begin{itemize}
\item [(P1)] $|N^+_D(X)| \geq d |X|$ and $|N^-_D(X)| \geq d |X|$ for all $X \subseteq V(D)$ with $1 \leq |X| < m$.
\item [(P2)] $e_D(X,Y) > 0 $ for all disjoint $X,Y \subseteq V(D)$ with $|X|=|Y|=m$.
\end{itemize}
Similarly, a graph $G$ is an $(n,d)$-expander if $|V(G)|=n$ and $G$ satisfies the following two conditions:
\begin{itemize}
\item [($P1'$)] $|N_G(X)| \geq d |X|$ for all $X \subseteq V(G)$ with $1 \leq |X| < m$.
\item [($P2'$)] $e_G(X,Y) > 0 $ for all disjoint $X,Y \subseteq V(G)$ with $|X|=|Y|=m$.
\end{itemize}
\end{definition}

\subsection{Expansion Properties of Random Graphs and Partitioning Expanders}
Random graphs and digraphs typically have strong expansion properties and this is quantified in the following lemma.
\begin{lemma}[Lemma 5.2 \cite{johannsen2013expanders}]\label{lemma:randomexpand} Let $d:\mathbb{N}\rightarrow\mathbb{R}^+$ satisfy $d\geq 3$. Then a graph $ G(n,7dn^{-1}\log n)$ is with high probability an $(n,d)$-expander.
\end{lemma}

For convenience of calculations we will often need the fact that an $(n,d)$-expander is a monotone property in $d$.
\begin{lemma}[Lemma 3.1 \cite{johannsen2013expanders}]
Let $n \in \mathbb{N}$ and $d,d_0 \in \mathbb{R}^+$ satisfy $3 \leq d_0 \leq d \leq n/6$.  Then every $(n,d)$-expander is also an $(n, d_0)$-expander.
\end{lemma}
Combining the above two lemmas, yields the following corollary.
\begin{corollary}\label{cor:expander}
Recall the constants from (\ref{eq:notation}).  There exists a $C>0$ such that
$G(n,p)$ is an $(n,n^{2 \eps})$-expander with probability at least $1 - n^{-\omega(1)}$, provided that $
p \geq C \left(\frac{\log n}{n^{\ell -1}} \right)^{1/\ell}.
$
\end{corollary}
\begin{remark}
Lemma 5.2 in \cite{johannsen2013expanders} does not include the probability bound, but the stronger statement above follows from an inspection of their proof.  An identical proof applies for a random directed graph.
\end{remark}

In many of our arguments, we will need the ability to partition the vertex set of an expander graph so that all vertex subsets expand ``nicely" into each partition set. The following lemma from \cite{johannsen2013expanders} allows for such a division. We state the slightly stronger form of their lemma used in \cite{montgomery2014embedding}.

\begin{definition}\label{def:expander}
For a graph $G$ and a set $W\subset V(G)$, we say $G$ \emph{$d$-expands} into $W$ if
\begin{enumerate} [$(P1)$]
\item$|N_G(X, W)|\geq d|X|$ for all $X\subset V(G)$ with $1\leq |X|<\left\lceil \frac{|W|}{2d}\right\rceil$, and,
\item $e_G(X,Y)>0$ for all disjoint $X,Y\subset V(G)$ with $|X|=|Y|=\left\lceil\frac{|W|}{2d}\right\rceil$.
\end{enumerate}
\end{definition}

\begin{lemma}[\cite{johannsen2013expanders}]\label{lemma:splitexpand} There exists an absolute constant $n_0\in \mathbb{N}$ such that the following statement holds. Let $k,n\in \mathbb{N}$ and $d\in \mathbb{R}^+$ satisfy $n\geq n_0$ and $k\leq \log n$. Furthermore, let $m,m_1,\ldots,m_k\in \mathbb{N}$ satisfy $m=m_1+\ldots+m_k$ and let $d_i:=\frac{m_i}{5m}d$ satisfy $d_i\geq 2\log n$ for all $i\in \{1,\ldots,k\}$. Then, for any graph $G$ which $d$-expands into a vertex set $W$, with $|W|=m$, the set $W$ can be partitioned into $k$ disjoint sets $W_1,\ldots, W_k$ of sizes $m_1,\ldots,m_k$ respectively, such that, for each $i$, $G$ $d_i$-expands into $W_i$.
\end{lemma}

Additionally, we have that for a large fixed set, there is sufficient expansion. The proof of the following lemma is a simple modification of that in \cite{johannsen2013expanders}.
\begin{lemma} \label{lemma:fixedexpand}
Let $\alpha >0$ be a constant.  There exists a $C >0$ such that for
a fixed set $X \subset [n]$ with $|X| \geq \alpha n$,
a graph $G(n,p)$ $n^{2\eps}$-expands into $X$ with probability at least $1 - n^{-\omega(1)}$, provided that $
p \geq C \left(\frac{\log n}{n^{\ell -1}} \right)^{1/\ell}.
$
\end{lemma}

\subsection{Simple integer partitioning lemmas}
The following lemmas will be used to decompose large cycles into the sum of smaller paths.
\begin{lemma} \label{lemma:sumrep}
Let $z \geq k^2$ be an integer.  Then for $m=\lfloor z/k\rfloor$ there exists $\{a_1, \dots, a_m\}$ such that $k \leq a_i \leq k+1$ and
$$
\sum a_i = z
$$
\end{lemma}
\begin{proof}
Write $z=mk+r$ where $r<k$ and observe that $m\geq k$. Take $a_1=\ldots=a_r=k+1$ and for every $1\leq i\leq m-r$ let $a_{r+i}=k$.
Clearly, $\sum_i a_i=mk+r=z$.
\end{proof}

\begin{lemma} \label{lemma:sumrep2}
Let $z \geq 3k^2/2$ be an integer.  Then there exists $\{a_1, \dots, a_t \}$ such that $k \leq a_i \leq k+1$ and
$$
\sum a_i = z
$$
Furthermore,
$$
\frac{|\{i: a_i = k\}|}{t}, \frac{|\{i: a_i = k+1\}|}{t} \geq \frac{1}{3}
$$
\end{lemma}

\begin{proof}
First, let us note that if all $a_i\in \{k,k+1\}$, then clearly $t \leq z/k$.  Now, for each $i \in [1, \lceil z/3k\rceil]$ let $a_i = k$ and for each $i \in [\lceil z/3k\rceil+1, \lceil 2n/3k\rceil]$ let $a_i = k+1$.  Observe that as $z - \lceil z/3\rceil - \lceil z/3k (k+1)\rceil \geq k^2$ one can complete the other $a_i$ using Lemma \ref{lemma:sumrep}. Clearly,
$$
\frac{|\{i: a_i = k\}|}{t} \geq \frac{\lceil z/3k\rceil }{z/k}\geq  \frac{1}{3}
$$
and
$$
\frac{|\{i: a_i = k+1\}|}{t} \geq \frac{\lceil z/3k\rceil }{z/k}\geq  \frac{1}{3}.
$$
\end{proof}

\subsection{Threshold for $H$-factors}
A main (and perhaps surprising) ingredient in our proofs is the result of Johansson Kahn and Vu \cite{johansson2008factors} which we state below. The surprising part is that in its simplest form, their result enables one to embed a factor of cycles in the desired probability, but the error probability is quite large if one hopes to apply a union bound. Before stating their result we need the following definitions.
\begin{definition}
For a graph $H$ on at least two vertices,
$$
d_1(H) = \frac{e(H)}{v(H)-1}
$$
\end{definition}
\begin{definition}
A graph $H$ is \emph{strictly balanced} if for any proper subgraph $H'$ of $H$ with at least two vertices, $d_1(H') < d_1(H)$.
\end{definition}
With these definitions in hand, the main result in \cite{johansson2008factors} relates a nearly optimal bound on the threshold for the appearance of $H$-factors in a random graph as a function of the density.
\begin{theorem}[Theorem 2.1 in \cite{johansson2008factors}]\label{thm:JKV}
	Let $H$ be a fixed graph which is strictly balance and suppose that $e(H)=m$. Then for any $C_1 > 0 $ there exists constant $C_2>0$ such that for $p\geq C_2 n^{-1/d_1(H)}\log^{1/m} n$, a graph $G(n,p)$ contains an $H$-factor with probability $1-n^{-C_1}$.
\end{theorem}

In the context of $2$-universality, we need the threshold for the appearance of $H$ factors when $H$ is a cycle.
\begin{corollary}\label{cor: JKV}
Let $H$ be a cycle of length $\ell$, then for any $C_1 >0$, there exists a $C_2 >0$ such that
$G(n,p)$ contains an $H$ factor with probability $1-n^{-C_1}$, provided that  $
p \geq C_2 \left(\frac{\log n}{n^{\ell -1}} \right)^{1/\ell}
$.
\end{corollary}

\subsection{Universality for Graphs of Bounded Cycle Length}
As we mentioned before, the error probability in Corollary \ref{cor: JKV} is too high for our purposes and it is also not clear yet how to use the result to embed a factor of cycles with different lengths. As a first step, we reduce our attention to the family of all graphs with max degree $2$ and with cycles of length at most some fixed constant $K$. Moreover, as we can obtain more precise results by forbidding short cycles, it would be convenient for us to define the following family:
Let $\ell$ and $K$ be two integers and let $\mathcal{H}^\ell(n,2,K)$ denote the family of all (unlabeled) graphs $H$ on $n$ vertices for which each connected component is an isolated vertex, an edge, or a cycle of length between $\ell$ to $K$.

The main advantage in bounding the largest cycle length is that now the number of such graphs $H$ (unlabeled) is polynomial in $n$ and it allows us to take a union bound while using Corollary \ref{cor: JKV}. Indeed, in order to uniquely determine $H$, it is enough to know how many cycles of length $1\leq t\leq K$ it has. Denoting this number by $X_t$ and observing that $\sum_{t=1}^KtX_t=n$ shows that
\begin{equation}\label{upperbound}|\mathcal H^{\ell}(n,2,K)|\leq \binom{n+K-1}{K-1}\leq n^{K}.\end{equation}

In the following lemma we show how to obtain a universality result for a typical $G(n,p)$ with respect to the family $\mathcal H^{\ell}(n,2,K)$.

\begin{lemma}\label{lemma:kuniversal}
For every $\ell,K\in \mathbb{N}$ there exists a constant $C$ such that a graph $G(n,p)$ is whp $\mathcal{H}^\ell(n,2,K)$-universal, provided that
$
p \geq C \left(\frac{\log n}{n^{\ell -1}} \right)^{1/\ell}
$.
\end{lemma}

\begin{proof}
Let $H \in \mathcal{H}^{\ell}(n,2,K)$. We wish to show that $\Pr[H\nsubseteq G(n,p)]\leq n^{-K^2}$, and then, by \eqref{upperbound} the proof is complete.
Recall that $X_{t}^H$ denotes the number of cycles of length $1\leq t\leq K$ in $H$ and trivially $\sum_{t=1}^KtX_t^H=n$. In particular, there exists $s\leq K$ such that $s X^H_{s} \geq \frac{n}{K}$. Fix such an $s$ and remove all cycles of length $s$ from $H$ to obtain an $H'$. We wish to embed $H$ in two steps:

  {\bf Step 1.} Embed $H'$ in $G_1=G(n,p/2)$. To achieve this, note that for every subset $S\subseteq [n]$ of size at least $n/K$ we have that $G_1[S]$ contains any cycle of length $\ell\leq t\leq K$. Indeed, for a fixed such subset, by Lemma \ref{lemma:cycles} we have that the probability of not having all such cycles is $\exp(-\omega(n))$.  Therefore, by applying the union bound over all such sets (at most $2^n$ possibilities) we obtain the claim. Now, assuming this, one can greedily embed all the cycles not of length $s$.

{\bf Step 2.} Completing the copy of $H'$ into a copy of $H$. Let $S$ denote the subset of vertices which do not belong to the copy of $H'$ from Step 1. Note that by applying Corollary \ref{cor: JKV} to $G_2=G(n,p/2)$ ($G_2$ is independent of $G_1$), there exists a $C>0$ so that with probability at least $1-n^{-K^2}$ we have that $G_2[S]$ contains a factor of cycles of length $s$ for $p \geq C \left(\frac{\log n}{n^{\ell -1}} \right)^{1/\ell}$.  Therefore, a union bound over all $\mathcal{H}^{\ell}(n,2,K)$ (at most $n^K$ graphs) assures us that we can complete $H$ for all $H\in\mathcal{H}^{\ell}(n,2,K)$.

Ultimately, as clearly $G_1\cup G_2$ can be coupled as a subgraph of $G(n,p)$, we obtain the desired result.
\end{proof}

\subsection{Connecting pairs of vertices with disjoint paths} \label{subsection:connectinglemmas}

Another key ingredient in our proof is the ability to find vertex disjoint paths of given lengths, connecting given pairs of vertices in expander graphs. This can be done using a beautiful argument of Montgomery \cite{montgomery2014embedding}. Unfortunately, in \cite{montgomery2014embedding} the corresponding lemmas are written for the case $p=polylog(n)/n$ so the path-lengths are bounded from bellow by $polylog(n)$ (note that in this regime the diameter of the graph is of order $\Omega(\log n/\log\log n)$ so one cannot expect to improve this by much). In our proofs we work with $p\geq n^{-1+\eps}$ (so the diameter of a typical $G(n,p)$ is a constant depending on $\eps$) and we need to be able to connect pairs of vertices with paths of constant lengths. Therefore, we had to rewrite the proofs, although the arguments are more or less identical to those in \cite{montgomery2014embedding}. We postpone the proofs to the Appendix.

\begin{lemma}[Non-spanning Connecting Lemma] \label{lemma:nonspanning} Recall the constants in (\ref{eq:notation}). Let $t$ be an integer and $D$ be a digraph on $s(n)\geq n/K$ vertices.  Let $\{(x_i,y_i)\}_{i=1}^t$ be a family of pairs of vertices in $V(D)$ with $x_i \neq x_j$ and $y_i \neq y_j$ for every distinct $i$ and $j$.  Suppose $k_1, \dots, k_t$ are integers with $5k \leq k_i \leq 1000k^2$ for each $1 \leq i \leq t$.  Assume that $W \subset V(D) \setminus \bigcup_{i=1}^t \{x_i,y_i\}$ is such that

\begin{enumerate}[(i)]
\item $|W| \geq s/2$
\item $
\sum_{i=1}^t k_i \leq \frac{7}{10}|W|.
$
\item $D$ $n^{\eps}$-expands into $W$
\end{enumerate}
Then there exist (internally) vertex-disjoint directed paths of length $k_i$ from $x_i$ to $y_i$ for all $i \in [1,t]$ with all vertices lying in $W$.
\end{lemma}

As the following version will not be used in a directed setting, for simplicity we prove it only for the undirected case. One can think about the directed version of it as a simple (but quite technical) exercise.

\begin{lemma}[Spanning Connecting Lemma]\label{lemma:spanningconnecting}
Let $G$ be a graph on $s(n) \geq n/K$ vertices.  Let $t, l$ be integers and $\{(x_i,y_i)\}_{i=1}^{t}$ be a family of pairs of vertices in $V(G)$ such that
\begin{enumerate}[(i)]
\item $1000 k^2 \leq l \leq 2K$
\item $x_i \neq x_j$ and $y_i \neq y_j$ for every distinct $i$ and $j$
\item $t (l-1) = |V(G) \setminus \bigcup_{i=1}^t \{x_i,y_i\}|$
\item $G$ $n^{\eps}$-expands into $V(G) \setminus \bigcup_{i=1}^t \{x_i,y_i\}$
\end{enumerate}

Then there exist $s(n)/t$ (internally) vertex-disjoint directed paths of length $l$ from $x_i$ to $y_i$ for all $i \in [1,t]$ with vertices in $V(G) \setminus \bigcup_{i=1}^t \{x_i,y_i\}$.
\end{lemma}

\section{Proof of Theorem \ref{thm:main}} \label{section:mainproof}
\begin{proof}
Given a graph $H\in \mathcal H^{\ell}(n,2)$ and an integer $s\in \mathbb{N}$, we use $X_s^H$ to designate the number of cycles of length exactly $s$ in $H$ (recall that we consider an isolated vertex as a cycle of length $1$ and an edge as a cycle of length $2$).
Now, let us partition $\mathcal H^{\ell}(n,2)=\mathcal H_1\cup \mathcal H_2$ where
$$\mathcal H_1:=\left\{H\in \mathcal H^{\ell}(n,2) \mid \sum_{s\leq K^{1/3}} sX_s^H \leq (1-1/K)n\right\}$$ and
$$\mathcal H_2:=\mathcal H^{\ell}(n,2)\setminus \mathcal H_1.$$

Our proof strategy is as follows. Let $q$ be such that $(1-q)^4=1-p$ (observe that $q\approx p/4$) and for each graph $H\in \mathcal H^{\ell}(n,2)$ let $S_H$ to denote the subgraph of $H$ consisting of all the components of $H$ of size at most $K$.  We let $S'_H$ denote the subgraph of $H$ consisting of all the components of size at most $K^{1/3}$.  We embed all the graphs in $\mathcal H^{\ell}(n,2)$ in three phases, where in each Phase $i\in \{1,2\}$ we expose an independent copy $G_i=G(n,q)$ and in Phase 3 we expose $G_3=G(n,q')$ where $q'=1-(1-q)^2$ (in particular, $q'\approx 2q\approx p/2$). In Phase 1 we show that $G_1$ whp contains copies of all the $S_H$, $H\in \mathcal H^{\ell}(n,2)$ and a fortiori all the $S'_H$. In Phase 2 we show that $G_2$ is whp such that for every $H\in \mathcal H_1$ there exists a copy of $S'_H$ in $G_1$ that can be changed to a copy of $H$ using edges of $G_2$. In Phase 3. we show that for every graph $H \in \mathcal H_2$, there exists a copy of $S_{H}$ in $G_1$ that can be changed to a copy of $H$ using edges in $G_3$.

Clearly, if we manage to do so, then $G_1\cup G_2\cup G_3$ can be coupled as a subgraph of $G\sim G(n,p)$ which yields the desired result.

Here are the formal details of our proof.

{\bf Phase 1.} Embedding all the $S_H$, $H\in \mathcal H^{\ell}(n,2)$.

Note that by Lemma \ref{lemma:kuniversal}, there is a $C_1>0$ such that a graph $G_1=G(n,q)$ is whp $\mathcal H^{\ell}(n,2,K)$ universal for $
p \geq C_1 \left(\frac{\log n}{n^{\ell -1}} \right)^{1/\ell}
$. In particular, for every $H\in \mathcal H^{\ell}(n,2)$ we have $S_H\in \mathcal H^{\ell}(n,2,K)$ and therefore one can find (and fix) a copy $T_H$ of $S_H$ in $G_1$. Furthermore, we fix a copy $T'_H \subset T_H$ of $S'_H$.  As there are at most $n^K$ distinct $S_H$, we can choose to fix at most $n^K$ distinct $T_H$ and $T'_H$.

{\bf Phase 2.} Embedding $\mathcal H_1$. It would be convenient for us in this phase to
refer to cycles of length larger than $K^{1/3}$ as \emph{long} cycles.  All other cycles are referred to as \emph{short} cycles.

For each distinct $T'_H$, specify an ordering of the vertices $V(G) \setminus V(T'_H)$ (identifying $V(G)$ with $[n]$ gives a natural ordering).
For each graph $H\in \mathcal H_1$, let $U_H:=V(H)\setminus V(S'_H)$. That is, $U_H$ is the set of all vertices lying on long cycles. It follows from the definition of $\mathcal H_1$ that $|U_H|\geq n/K$ for all such $H$.
Our general strategy is to show that $G_2[V(G)\setminus V(T'_H)]$ has ``good" expansion properties for all $H$ and then to use one of the Connecting Lemmas introduced before to claim that such a graph contains any factor of long cycles.

Formally, we act as follows. Let $H \in \mathcal{H}_1$.  Enumerate the long cycles from $1$ to $m :=m(H)= \sum_{s> K^{1/3}} X_s^H$.
Consider a long cycle $C_i$ of length $\ell_i > K^{1/3}$.  As our choice of $K$ and $k$ entails that $K^{1/3} > \frac{3(1000k^2)^2}{2}$, by Lemma \ref{lemma:sumrep}, we can find a representation of $\ell_i$ as $\sum_{j=1}^{t_i} \alpha_{ij} = \ell_i$ with each $\alpha_{ij}$ in $[1000k^2,1000k^2+1]$.  Furthermore,
$
|\{i: a_i = 1000k^2\}| \geq \frac{t_i}{3}
$
and
$
|\{i: a_i = 1000k^2+1\}| \geq \frac{t_i}{3}.
$
Fix such a representation for all long cycles.

A simple yet important observation to keep in mind is the following. Suppose that we fix a subset of vertices $A_i = \{a_{i1}, \dots , a_{it_i} \} \subset V(G)\setminus V(T_H)$ and define the set $B_i = \{b_{i1}, \dots , b_{it_i} \}$ with $b_{ij} = a_{i(j+1)}$ for $1 \leq j \leq t_i - 1$ and $b_{it_i} = a_{i1}$. Moreover, assume that we are able to find internally vertex-disjoint paths $P_{ij}$, $1\leq j\leq t_i$ such that $P_{ij}$ connects $a_{ij}$ to $b_{ij}$  and is of length $\alpha_{ij}$ for all $j$. Then, the cycle $P_{i1}P_{i2}\ldots P_{it_i}$ is a copy of $C_i$.

Our goal now is to fix such sets $A_i$ and $B_i$ for every long cycle $C_i$ and to find the corresponding paths simultaneously for all cycles. Then, we need to argue that one can do so for all of our target graphs $H$. To this end, let us first observe that for every $i$ the potential set $A_i$ is of size $|A_i|=t_i$. Therefore, in order to embed all long cycles, one needs to fix $t:=t(H)=\sum_i |A_i|=\sum_it_i$ vertices. Let $A:=A_H \subset V(G)\setminus V(T'_H)$ be the set consisting of the $t$ smallest vertices according to the vertex ordering. Clearly (and crucially), by the way we defined $A$ we have at most $n^{K+1}$ distinct pairs $(T'_H, A_H)$ for $H \in \mathcal{H}^\ell(n,2)$.

Next, fix a partition of $A$ into disjoint sets $A_1\cup \ldots \cup A_{m}$, label $A_i=\{a_{i1},\ldots, a_{it_i}\}$ for all $i$ and define the $B_i$ as in the discussion above. Note that each $A_i$ is going to be part of (a copy of) the cycle $C_i$. Next, we wish to label all pairs $(i,j)$ having the same $\alpha_{ij}$ and therefore we define $$I_{1} := \{(i,j): \alpha_{ij} = 1000k^2\}$$ and $$I_2 := \{(i,j): \alpha_{ij} = 1000k^2+1\}.$$  Furthermore, we let $A^1 = \{a_{ij}\mid (i,j)\in I_1\}$ to be those vertices that are the initial points for paths of length $1000k^2$ and similarly, $A^2 = \bigcup_{(i,j) \in I_{H,2}} a_{i,j}$ will be the initial vertices for paths of length $1000k^2+1$.  We then naturally partition $B=B^1\cup B^2$ in such a way that $B^i$, $i\in\{1,2\}$ contains all the (potential) endpoints of paths starting at $A^i$.

Recall that there are at most $n^{K+1}$ distinct pairs $(T'_H, A_H)$ as $H$ ranges over all of $\mathcal{H}^\ell(n,2)$.  Therefore, by applying Lemma \ref{lemma:fixedexpand} and the union bound, a graph $G_2=G(n,q)$ with high probability $n^{2\eps}$ expands into $W_H:= V(G) \setminus (T_H \cup A_H)$ for all $H$.  It thus remains to show that one can complete the argument in this pseudorandom setting by restricting our attention to a specific $H \in \mathcal{H}_1$ (as there is no more probability involved).

For an $H \in \mathcal{H}_1$, one can easily check that $s_1:= \sum_{(i,j) \in I_{1}} (\alpha_{ij}-1) \geq \frac{n}{4K}$ and $s_{2} := \sum_{(i,j) \in I_{2}} (\alpha_{ij}-1) \geq \frac{n}{4K}$.  Thus, by Lemma \ref{lemma:splitexpand} we can partition $W_H = W^{1} \cup W^{2}$ with $|W^{1}| = s_{1}$, $|W^{2}| = s_{2}$, and $G_2$ $n^{\eps}$-expands into both $W^{1}$ and $W^{2}$.
As $A^1, B^1$ and $W^{1}$ satisfy the conditions of Lemma \ref{lemma:spanningconnecting}, one can find vertex disjoint paths of length $1000k^2$ from $a_{ij}$ to $b_{ij}$ for $(i,j) \in I_{1}$ that cover all of $W^{1}$.  Applying the same argument for $A^2, B^2$ and $W^{2}$, we have now found vertex disjoint paths from $a_{ij}$ to $b_{ij}$ of length $1000k^2+1$ for all $(i,j) \in I_{2}$ which cover all of $W^{2}$. These paths form the desired cycles needed to extend $T'_H$ to a full embedding of $H$.

{\bf Phase 3.} Embedding $\mathcal H_2$. In this phase it will be convenient for us to refer to cycles of length more than $K$ (as opposed to $K^{1/3}$ in Phase 2) as \emph{long} cycles and the remaining cycles as \emph{short}. Our strategy here is similar in nature to the one in Phase 2. That is, we wish to delete all the long cycles from a graph $H\in \mathcal H_2$, to embed the smaller ones, and then to complete the embedding using some pseudorandom properties. To demonstrate the difficulty in this case, imagine that $H$ consists of $(n-\log n)/3$ vertex disjoint triangles and one cycle of length $\log n$. In this scenario, the strategy from the previous phase will fail as one cannot expect a random graph on $\log n$ vertices to have expansion properties in our edge density. In order to overcome this obstacle, we observe that if ``most" of the vertices belong to short cycles, then ``many" of them belong to cycles of the \emph{exact} same length. Then, we ``cut" each long cycle into vertex disjoint cycles (actually paths, but we may add one edge to each path in order to turn it into a cycle) of this particular length plus some isolated vertices if there are some divisibility issues to obtain a new graph $H'\in \mathcal H^{\ell}(n,2,K)$. Finally, by defining an appropriate auxiliary digraph, we show that one can complete a copy of $H'$ (we fixed such a copy in Phase 1) into a copy of $H$ using some pseudorandom properties.

Here is a formal description of our embedding scheme. First, given a graph $H\in \mathcal H_2$, we wish to specify the special graph $H'\in \mathcal H^{\ell}(n,2,K)$ along with some parameters that will be of use later. By definition, for every $H \in \mathcal{H}_2$, there is a unique smallest index $u:=u(H)$ such that
$$
u X_{u}^H \geq \frac{n}{2K^{1/3}}.
$$
We wish to define $H'$ by replacing long cycles with cycles of length $u$ and a few isolated vertices. In order to do so,
let us fix an enumeration of the long cycles from $1$ to $m :=m(H) =\sum_{s > K} X_s^H$.
For every long cycle $C_i$ in $H$ let us write
$
|C_i|=\gamma_iu+\beta_i,
$ with $\gamma_i$ being a positive integer and $0\leq \beta_i<u$. Define $H'$ by replacing each $C_i$ by $\gamma_i$ vertex-disjoint cycles of length $u$ each and $\beta_i$ isolated vertices. Clearly, we have
\begin{enumerate}[(i)]
\item $H' \in \mathcal{H}^\ell(2,n,K)$, and
\item $X_s^{H'} = X_s^{H}$ for $s \in [2,K] \setminus \{u\}$, and
\item $X_{u}^{H'} = X_{u}^{H}+ \sum_i\gamma_i$, and
\item $X_1^{H'}=X_1^H+\sum_i\beta_i$.
\end{enumerate}

Second, note that as $H'\in \mathcal H^{\ell}(2,n,K)$, it follows from Phase 1 that there exists a copy $T_{H'}$ of $H'$ in $G_1$ (we fixed an arbitrary such copy).  Our goal is to show that $T_{H'}$ can be turned into a copy of $H$ for all $H\in \mathcal H_2$ using edges of $G_3=G(n,q')$ and some pseudorandom properties. For a small technical reason, it will be convenient for us to further write $G_3=G_4\cup G_5$ where $G_4$ and $G_5$ are two independent copies of $G(n,q)$.  

We act as follows. For all distinct $T_{H'}$ we fix an ordering on $V(G) \setminus V(T_{H'})$ and an ordering for the cycles of each length from one to $K$. Then, consider a graph $H \in \mathcal{H}_2$ and define an auxiliary digraph $D:=D(H)$ in the following way:  let us fix an edge $z_i=s_it_i$ on each cycle of length $u$ in $T_{H'}$ (if $u$ is $1$, then for convenience we write $z_i=s_it_i$ where $s_i=t_i$) and define $Z$ to be the collection of all $z_i$. Moreover, let $A:=A(H)$ to be the set consisting of the $\sum_i\beta_i$ smallest (according to the fixed labeling) isolated vertices in $T_{H'}$. Let $V(D)=Z\cup A$ and the set of arcs is defined as follows: 
\begin{enumerate} [$(a)$]
\item for $z=st,z'=s't'\in V(D)\setminus A$ we add and arc $zz'\in E(D)$ if and only if $ts'\in E(G_4)$. 
\item For pairs $zz'\in V(D)$ with $z=st\in Z$ and $z'=v\in A$ we add $zz'\in E(D)$ if $tz'\in E(G_4)$ and $z'z\in E(D)$ if $z's\in E(G_4)$.
\item For pairs $zz'\in A$ we do the following. As $A\subseteq [n]$ we have a natural ordering on $A$. For $z<z'$ we add $zz'\in E(D)$ if and only if $zz'\in E(G_4)$ and $z'z\in E(D)$, if and only if $zz'\in E(G_5)$.
\end{enumerate}

Clearly, $D=D(|V(D)|,q)$ (and this is the only reason we split $G_3=G_4\cup G_5$; so that we can borrow some desired properties without reproving them).

The main observation here is the following. Consider a long cycle $C_i$ (recall that $|C_i|=\gamma_iu+\beta_i$). Suppose that $\sum_{j=1}^{t_i}(\alpha_{ij}+1)=\gamma_i+\beta_i+2$, $A_i\subset A$ is of size $\beta_i$, $Z_i\subset Z$ is of size $t_i-\beta_i$ and relabel $A_i\cup Z_i=\{s_1,\ldots,s_{t_i}\}$. Suppose we can find (in $D$) internally vertex-disjoint paths $P_j$, $1\leq j\leq t_i$, such that
\begin{enumerate}
\item $P_j$ connects $s_j$ to $s_{j+1}$ for all $j$ (where $t_i+1=1$), and
\item $P_j$ is of length exactly $\alpha_{ij}$, and
\item $V(P_j)\setminus \{s_j,s_{j+1}\}\subseteq Z$.
\end{enumerate}

Then, taking the cycle $C$ (in $G_3$) obtained from the cycle $P_1P_2\ldots P_{t_i}$ (in $D$) by replacing each $z$ with the corresponding cycle of length $u$ (and deleting the edge $z$), we obtain
$$|C|=(\sum^{t_i}_{j}\alpha_{ij}-2)u+(t_i-\beta_i)u+\beta_i=\left(\sum_j(\alpha_{ij}+1)-2-\beta_i\right)u+\beta_i=|C_i|.$$

Our goal now is to fix such sets $A_i$ and $Z_i$ for every long cycle $C_i$ and to find the corresponding paths simultaneously for all cycles. Moreover, we need to argue that we can do it for all $D(H)$. With this goal in mind, we perform the following. For each long cycle $C_i$ let us
fix a representation
$$
\sum_{j=1}^{t_{i}} (\alpha_{ij}+1) = \gamma_i+\beta_i+2
$$
with $100k-1\leq \alpha_{ij} \leq 100k$ (this is possible by Lemma \ref{lemma:sumrep} and the fact that $\gamma_i+\beta_i+2\geq |C_i|/u \geq K^{2/3} \geq (100k)^2$). Let $Z':=Z'(H)$ be the set consisting of the smallest (according to the fixed ordering) $\sum_i(t_i-\beta_i)$ elements $z_i\in Z$ and let $W=V(D)\setminus (A\cup Z')$.
Expose the edges of $G_3$ and observe from Lemma \ref{lemma:fixedexpand} that the digraph $D$ $n^{2\eps}$-expands into $W$ with probability $1-n^{-\omega(1)}$ (and therefore it holds for all distinct triples $(D(H),A\cup Z',W)$, as there are at most $n^{K+3}$ many such triples).

We can now work in a pseudorandom setting (that is, in $D$) to show how to modify $T_{H'}$ into a copy of $H$. To this end,
let us label the vertices $A= \{a_{ij}\mid 1\leq i\leq m, 1\leq j\leq \beta_i\}$ and $Z' = \{a_{ij}\mid 1 \leq i \leq m, \beta_i + 1 \leq j \leq t_i\}$. For each $i$ define $A_i=\{a_{ij} \mid 1\leq j\leq \beta_i\}$ and $Z_i:=\{a_{ij} \mid 1\leq j\leq t_i\}$. Moreover, let us set $B_i=\{b_{ij} \mid 1\leq j\leq t_i\}$ by defining $b_{ij} = a_{i(j+1)}$ for all $1 \leq i \leq m$ and $1 \leq j \leq t_i$ (where $t_i+1=1$ for all $i$). In order to complete the proof, we want to apply Lemma \ref{lemma:nonspanning} to $A\cup Z'$ and $B:= \bigcup B_i$ (in $D$) in order to connect each $a_{ij}$ to $b_{ij}$ with a path of length $\alpha_{ij}$ using vertices of $W$ (such that all paths are internally vertex-disjoint). Thus it only remains to verify that all the assumptions in Lemma  \ref{lemma:nonspanning} are satisfied.

Note that, by definitions of $\mathcal H_2$ and $D$, the total number of vertices in the paths, $\sum_{i,j}\alpha_{ij}$, is less than
$\frac{2n}{K}$. Moreover, as $uX^H_u\geq n/2K^{1/3}$ and $u\leq K^{1/3}$ it follows that $|V(D)|\geq n/2K^{2/3}\geq 100n/K$. In addition, as $\alpha_{ij}\geq 100k-1\geq 99k$ for all $i,j$, it follows that $|A\cup Z'|\leq |V(D)|/99k$ and therefore, $\sum_{i,j}\alpha_{ij}<7|W|/10$ (recall that $W=V(D)\setminus (A\cup Z')$.

Finally, as all the assumptions of Lemma \ref{lemma:nonspanning} are met, the proof is complete.
\end{proof}

\bibliographystyle{abbrv}
\bibliography{2Universalitybib}

\appendix
\section{Preparatory Lemmas} \label{section:preparatorylemmas}

\begin{definition}
For a bipartite graph $G(U,V)$ with $|V| = k |U|$, $|U|$ disjoint copies of $K_{1,k}$ (a star) is a \emph{$k$-matching} from $U$ to $V$.
\end{definition}
\subsection{Star Matchings}
We will make use of the following simple observation.
\begin{lemma} \label{lemma:starmatching}
Consider two subsets $A, X \subset V(G)$ with $|V(G)|= n$, $|A| = n^{1-\eps/3}$, and $|X| \geq \frac{n}{10 K^2}$.  If $A$ $n^{\eps/3}$-expands into $X$ then there exists a set $X' \subset X$ such that there is an $n^{\eps/6}$-matching from $A$ to $X$.
\end{lemma}

\begin{proof}
If we can show that for all $B \subset A$,
$$
|N(B, X)| \geq |B| n^{\eps/6}
$$
then we can greedily create such a matching.

It remains to verify the condition.  There are two cases to consider:
\begin{itemize}
\item For $B \subset A$ with $|B| \leq \frac{|W|}{2 n^{\eps/3}}$, the condition follows from our definition of expansion since
$$
N(B,X) \geq n^{\eps/3} |B| \geq n^{\eps/6} |B|.
$$

\item If $|B| > \frac{|W|}{2 n^{\eps/3}}$ then
$$
N(B,X) \geq |W|/2 \geq n^{1-\eps/3} n^{\eps/6} \geq n^{\eps/6} |B|.
$$
\end{itemize}
\end{proof}

\subsection{Flexible Bipartite Graph}
We will later employ a bipartite graph that has robust properties in the following sense.
\begin{lemma}[Lemma 2.8 \cite{montgomery2014embedding}]\label{lemma:smallbipartite}
There is a constant $n_0$ such that for every $n \geq n_0$ with $3 | n$, there exists a bipartite graph $H$ on vertex classes $X$ and $Y \cup Z$ with $|X| = n$, $|Y| = |Z| = 2n/3$, and maximum degre 40, so that the following is true.  Given any subset $Z' \subset Z$ with $|Z'|=n/3$ there is a matching between $X$ and $Y \cup Z'$.
\end{lemma}

\section{Proof of Connecting Lemmas} \label{section:proofconnecting}
The following connecting lemmas will allow us to link pairs of vertices together with paths of constant length under the assumptions of sufficient expansion in the underlying graph.  We need both directed and undirected versions.  Additionally, a spanning version of this lemma will be able to incorporate all the vertices in a certain set using absorbers.  The proofs of these lemmas will follow those of Montgomery in \cite{montgomery2014embedding}.  We make modifications where necessary.
\subsection{Non-spanning Connecting Lemma}
We will make use of the following simple lemma.

\begin{lemma} \label{lemma:divide}
Let $X$ and $Y$ be disjoint sets of vertices such that for every $y \in Y$ there exists a (directed) path from some $x \in X$ to $y$.  Then for $k \in \mathbb{N}^+$, there exists a non-empty set $X' \subset X$ of size at most $\lceil|X|/k\rceil$ such that there is a set $Y' \subset Y$ of size at least $\lfloor |Y|/k \rfloor$ so that for every $y \in Y'$ there is a (directed) path from some $x \in X'$ to $y$.
\end{lemma}

\begin{proof}
Partition $X$ into $V_1, \dots V_k$ non-empty sets of size at most $\lceil |X|/k \rceil$.  Assume to the contrary that each set $V_1, \dots, V_k$ has paths to fewer than $\lfloor |Y|/k \rfloor$ vertices of $Y$.  Thus, in total, there are fewer than $|Y|$ vertices of $Y$ that are the terminating points of paths from $X$, which contradicts our assumptions.
\end{proof}

Now, we introduce the lemma that allows the connection of a single pair of vertices.

\begin{lemma}\label{lemma:singlepair}
Let $m:= n^{1-\eps/3}$ and $G$ be a digraph of size $|V(G)| = s(n)\geq 3n/10K$ such that any set $A \subset V(G)$ with $|A| = m$ satisfies $|N(A)| \geq (1-1/100)s$.
Furthermore, for disjoint sets $A$ and $B$ of size $m$, $e(A,B) > 1$.  Let $k := \frac{12}{\eps}$, and consider a set of integers $\{k_1, \dots, k_{3m}\}$ with $k \leq k_i \leq 1000k^2$.
If $X = \{x_1, \dots, x_{3m}\}$ and $Y = \{y_1, \dots, y_{3m}\}$, $W$ are three subsets such that
$$
|X| = |Y| =3m.
$$
$W$ is disjoint from $X \cup Y$, and $|W| \geq \frac{s}{10}$, then there exists a directed path of length $k_i$ from $x_i$ to $y_i$ for some $i \in [1,3m]$ with vertices in $W$.
\end{lemma}

\begin{proof}
Partition $W = W_X \sqcup W_Y$ so that $W_X = W_Y = |W|/2$.  Choose a subset $X_0 \subset X$ of size $m$.  By the assumption that $|N(X_0)| \geq (1-1/100)s$, $N(X,W_X)$ is of size at least $|W|/3$. Thus, we can fix a set $X_1 \subset N^+(X_0,W_X)$ of size $s/30$.  Note that for every $x \in X_1$, there exists a path of length one from some element of $X_0$ to $x$.  By Lemma \ref{lemma:divide}, there exists a subsets $X_0^1 \subset X_0$ and $X_1' \subset X_1$ such that
$$
|X_0^1| \leq \left \lceil \frac{m}{n^{\eps/4}} \right\rceil \leq n^{1-\eps/2},
$$
$$
|X_1'| = m
$$
and for every $x_1 \in X_1'$ there exists a directed path of length one from some element $x \in X_0^1$ to $x_1$.

Again, by assumption,
$$
|N^+(X_1', W_X)| \geq \frac{|W|}{2} - \frac{s}{100} \geq \frac{s}{30}
$$
so there exists an $X_2 \subset N^+(X_1',W_X) \setminus X_1'$ with $|X_2|=s/30$.  Another application of Lemma \ref{lemma:divide} yields subsets $X_0^2 \subset X^1$ and $X_2' \subset X_2$ such that
$$
|X_0^2| \leq \left \lceil \frac{n^{1-\eps/2}}{n^{\eps/4}}\right \rceil \leq n^{1-7\eps/10},
$$
$$
|X_2'| = m
$$
and there is a directed path of length two from $X_0^2$ to every element of $X_2'$.

After the $i$-th iteration of the above procedure, we are left with subsets $X_0^i \subset X_0^{i-1}$ and $X_i' \subset X_i$ such that
$$
|X_0^i| \leq \lceil n^{1-\eps/2 - (i-1)\eps/5} \rceil,
$$
$$
|X_i'| = m
$$
and there is a directed path from $X_0^i$ to every vertex in $X_i'$.  We can iterate a constant number of times as
$$
|N^+(X_{i-1}',W_X) \setminus \cup_{j=1}^{i-1} X_j'| \geq \frac{|W|}{2}-(i-1)m \geq \frac{s}{30}
$$
so we can always find a $X_i \subset N^+(X_{i-1}',W_X) \setminus \cup_{j=1}^{i-1} X_j'$ of size $s/30$.
By the $k/2$-th iteration,
$$
|X_0^{k/2}| \leq \lceil n^{-\eps/2} \rceil =1
$$
and since $X^{k/2}$ is non-empty, $|X^{k/2}| = 1$ or $X^{k/2} = \{ x_i\}$ for some $1 \leq i \leq m$.  Furthermore, there exists a directed path of length $k/2$ from $x_i$ to every vertex in $X_{k/2}'$ and $|X_{k/2}'| = n^{1-\eps/3}$.  Finally, iterate until we have a path of length $\left \lceil \frac{k_i}{2} \right \rceil -1$.

As $3m - (\frac{3m}{2}+1) \geq m$, we can repeat the entire procedure $\frac{3m}{2}+1$ times, each time removing the single vertex in $X$ from which all the paths originated in the previous round.  This generates a set of $\frac{3m}{2}+1$ indices
$$
I = \{i_1, \dots, i_{\frac{3m}{2}+1}\} \subset [1,3m]
$$
and $\frac{3m}{2}+1$ not necessarily disjoint sets
$$
\{Z_{i_1}, Z_{i_2}, \dots, Z_{i_{\frac{3m}{2}+1}} \}
$$
such that $Z_{i_j} \subset W_X$ and there are directed paths of length $\left \lceil \frac{k_{i_j}}{2} \right \rceil -1$ from $x_{i_j}$ to every vertex in $Z_{i_j}$ for all $1 \leq j \leq \frac{3m}{2}+1$.

Applying the same arguments to $Y$ and $W_Y$ with in-neighborhoods instead of out, we can find a set of $\frac{3m}{2}+1$ indices
$$
I' = \{i_1', \dots, i_{\frac{3m}{2}+1}'\}
$$
and not necessarily disjoint sets
$$
\{Z_{i_1'}', Z_{i_2'}', \dots Z_{i_{\frac{3m}{2}+1}}'\}
$$
such that $Z_{i_j'}' \subset W_Y$ and from every vertex in $Z_{i_j'}$ there exists a directed path of length $\left \lfloor \frac{k_{i_j'}}{2}\right \rfloor$ to $y_{i_j}$.

Since $|I|+|I'| > 3m$, there exists an index in their intersection, say $i$.  By the assumption of the lemma, there exists an edge from some $z \in Z_i$ to some $z' \in Z_{i}'$.  Concatenating the path of length $\left \lceil \frac{k_i}{2} \right \rceil -1$ from $x_i$ to $z$, the edge from $z$ to $z'$, and the path of length $\left \lfloor \frac{k_{i}}{2}\right \rfloor$ from $z'$ to $y_i$ yields a path of length $k_i$ from $x_i$ to $y_i$.
\end{proof}

Building off the previous lemma, we connect all the pairs of vertices simultaneously.
\begin{lemma}[Non-spanning Connecting Lemma]
Recall the constants in (\ref{eq:notation}). Let $t$ be an integer and $D$ be a digraph on $s(n)\geq n/K$ vertices.  Let $\{(x_i,y_i)\}_{i=1}^t$ be a family of pairs of vertices in $V(D)$ with $x_i \neq x_j$ and $y_i \neq y_j$ for every distinct $i$ and $j$.  Suppose $k_1, \dots, k_t$ are integers with $5k \leq k_i \leq 1000k^2$ for each $1 \leq i \leq t$.  Assume that $W \subset V(D) \setminus \bigcup_{i=1}^t \{x_i,y_i\}$ is such that

\begin{enumerate}[(i)]
\item $|W| \geq s/2$
\item $
\sum_{i=1}^t k_i \leq \frac{7}{10}|W|.
$
\item $D$ $n^{\eps}$-expands into $W$
\end{enumerate}
Then there exist (internally) vertex-disjoint directed paths of length $k_i$ from $x_i$ to $y_i$ for all $i \in [1,t]$ with all vertices lying in $W$.
\end{lemma}

\begin{proof}
By Lemma \ref{lemma:splitexpand}, we can partition $W$ into disjoint sets $W_1, W_2, \dots, W_k, Z_1, \dots, Z_k$ and $U$ with $|W_i| = |Z_i| = |W|/20k$ and $|U| = 9|W|/10$.  Furthermore, $D$ $n^{2\eps/3}$-expands into each $W_i$ and $U$.

Set $m:= n^{1-\eps/3}$.
For any subset $A \subset V(D)$ with $|A| = m$, since $G$ $n^{\eps}$-expands into $W$ and there are no edges between $D(V)\setminus (N(A) \cup A)$, it must be that $N(A) \geq s - 2m \geq (1-1/100)s$.
We begin by applying Lemma \ref{lemma:singlepair} as many times as possible to generate paths from $X$ to $Y$ using vertices in $U$. In each iteration, we ignore the vertices of $x$ that have paths to their corresponding $y$'s and remove the vertices and edges that have been used in those paths.  Abusing notation, we let $s'$ denote the number of vertices at the end of each step.  After $i$ iterations, the remaining vertices $U' \subset U$ has size at least
$$
|U'| \geq \frac{9|W|}{10} - \frac{7|W|}{10} \geq \frac{s}{10} \geq \frac{s'}{10},
$$
by $n^{\eps/3}$-expansion, for any subset $A \subset U'$,
$$
|N(A) \setminus (U \setminus U')| \geq s' - 2m \geq (1-1/100)s',
$$
and
$$
s' \geq s - \frac{7|W|}{10} \geq 3n/10K.
$$
Thus, the only stipulation of Lemma \ref{lemma:singlepair} that is eventually violated is the number of unpaired vertices in $X$ and $Y$.

At this stage, we have linked many of the $x \in X$ and $y\in Y$ by directed paths of the correct length.  The remaining sets will be denoted as $X_1$ and $Y_1$.  By $D$'s expansion into $W_1$ and Lemma \ref{lemma:starmatching}, we can find an $n^{\eps/6}$ star matching of $X_1$ and some $W_1' \subset W_1$ with edges directed towards $W_1'$.  Similarly, there exists a $n^{\eps/6}$ star-matching from $Y_1$ to $Z_1' \subset Z_1$ with edges directed toward $Y_1$.  We now can apply Lemma \ref{lemma:singlepair} to connect the corresponding vertices of $W_1'$ and $Z_1'$ by paths of length $k_i-2$ yielding paths of length $k_i$ from $x_i$ and $y_i$.  We can continue until there are $3m$ unconnected vertices in $X_1$ and $Z_1$.  This corresponds to at most $m/n^{\eps/6}$ unconnected vertices in the original sets $X$ and $Y$.  Continuing in this way, we will need at most $k$ iterations to connect all the vertices.  Note that since
$$
\sum_{i=1}^t k_i \leq \frac{7|W|}{10},
$$
the number of available vertices at each step is greater than
$$
\frac{9|W|}{10} - \frac{7|W|}{10} \geq \frac{s}{10}.
$$
This observation along with the constraint that $k_i \geq 5k$ guarantees that Lemma \ref{lemma:singlepair} applies at each iteration.
\end{proof}

\begin{remark}
Observe that the above proofs go through just as smoothly if we use graphs instead of digraphs.  Thus, when necessary, we will use the non-oriented version of Lemma \ref{lemma:nonspanning} without further comment.
\end{remark}
\subsection{Absorbers}
\begin{definition}
In a graph $G$, we say that $(R,r,s)$ is an \emph{absorber} for a vertex $v \in V(G)$ if $\{r,s\} \subset R \subset V(G)$, $r \neq s$, and if there is both an $r,s$-path in $G$ with vertex set $R$ and an $r,s$-path in $G$ with vertex set $R \cup v$.  We refer to $|R|$ as the \emph{size} of the absorber, and call the vertices $r,s$ the \emph{ends} of the absorber.
\end{definition}

Lemma \ref{lemma:nonspanning} roughly means that we can cover $7/10$ of the vertices in a set with paths.  For the spanning version, we will use absorbers to boost this to the entire set.  We first show that there exist many absorbers for a vertex.  Again, the arguments follow those of Montogmery \cite{montgomery2014embedding}.

\begin{lemma}[Absorbers]\label{lemma:absorbers}
Let $G$ be a graph of size $s(n) \geq n/K$ and $W$ a subset such that $G$ $n^{\eps}$-expands into $W$.  Then, given any set $A \subset V(G)\setminus W$ with $|A| \leq |W|/3500 k^2$, we can find 40 edge disjoint absorbers of size $18k^2 + 2$ in $G[W]$.
\end{lemma}

\begin{proof}
By Lemma \ref{lemma:splitexpand}, we can partition $W$ as $W_1 \sqcup W_2 \sqcup W_3$ so that $|W_1| = |W_2| = |W_3| = \frac{s}{3}$ and $G$ $n^{\eps}$-expands into $W_1, W_2,$ and $W_3$.  Since, $|A| \leq |W_1|/100$, by Hall's theorem, there exists an $80$-matching from $A$ into $W_1$.  For $v \in A$, we denote its neighbors under such a matching by $B_v$.

We intend to create 40 absorbers for every $v \in A$, each incorporating a different pair of vertices in $B_v$.  The strategy will be to create paths using $W_2$ and $W_3$.  Ultimately, there will be $40|A|$ absorbers utilizing $40 |A| (18k^2 + 2)\leq \frac{7|W_2|}{10} = \frac{7 |W_3|}{10}$ so all the absorbers can be found disjointly and simultaneously by appealing to Lemma \ref{lemma:nonspanning}.

For clarity, we restrict our explanation to the construction of a single absorber for $v \in A$ using the pair $\{x_0, y_1\} \in B_v$.  By Lemma \ref{lemma:nonspanning}, we can find a path of length $6k+1$  from $x_0$ to $y_1$ with interior vertices in $W_2$.  Let $Q$ denote the path $x_0 x_1 \dots x_{3k} y_0 y_{3k} y_{3k-1} \dots y_1$.  Since, $G$ $n^{\eps}$-expands into $W_3$, again by Lemma \ref{lemma:nonspanning}, we can find $3k$ disjoint paths $P_i$ of length $6k-2$ from $x_i$ to $y_i$.

Let $R = \cup_i V(P_i)$.  For $k$ even, the following two $x_0,y_0$-paths have vertex set $R$ and vertex set $R \cup v$.
$$
x_0 x_1 P_1 y_1 y_2 P_2 x_2 x_3 \dots y_{3k} P_{3k} x_{3k}
$$
and
$$
x_0 v y_1 P_1 x_1 x_2 P_2 y_2 y_3 \dots x_{3k} P_{3k} y_{3k} y_0
$$
(see Figure \ref{image:absorber}).
\begin{figure}\label{image:absorber}
\begin{center}
\begin{tikzpicture}

\tikzset{vertex/.style = {shape=circle,draw,scale=.5,fill=black}}
\tikzset{edge/.style = {->,> = latex'}}
\tikzset{edge1/.style= {->,> = latex', dashed}}
\tikzset{edge2/.style= {-}}
\tikzset{edge3/.style= {-, line width=2pt}}
\node[vertex] (x0) at  (0,0) {};
\node[vertex] (v) at  (1,0) {};
\node[vertex] (y1) at (2,0) {};
\node[vertex] (x1) at (4,0) {};
\node[vertex] (x2) at (5,0) {};
\node[vertex] (y2) at (7,0) {};
\node[vertex] (y3) at (8,0) {};
\node[vertex] (x3) at (10,0) {};
\node[vertex] (y0) at (11, 0) {};

\draw[edge2] (x0) to (v);
\draw[edge2] (v) to (y1);
\draw[edge3] (y1) to (x1);
\draw[edge2] (x1) to (x2);
\draw[edge3] (x2) to (y2);
\draw[edge2] (y2) to (y3);
\draw[edge3] (y3) to (x3);
\draw[edge2] (x3) to (y0);

\draw[edge2] (x1) to[bend left=45] (x0);
\draw[edge2] (y1) to[bend left=45] (y2);
\draw[edge2] (x3) to[bend left=45] (x2);
\draw[edge2] (y3) to[bend left=45] (y0);

\node[text width=.25cm] at (0,.35) {$x_0$};
\node[text width=.25cm] at (1,.35) {$v$};
\node[text width=.25cm] at (2,-.35) {$y_1$};
\node[text width=.25cm] at (4,.35) {$x_1$};
\node[text width=.25cm] at (5,.35) {$x_2$};
\node[text width=.25cm] at (7,-.35) {$y_2$};
\node[text width=.25cm] at (8,-.35) {$y_3$};
\node[text width=.25cm] at (10,.35) {$x_3$};
\node[text width=.25cm] at (11,-.35) {$y_0$};

\node[text width=.25cm] at (3,.35) {$P_1$};
\node[text width=.25cm] at (6,.35) {$P_2$};
\node[text width=.25cm] at (9,.35) {$P_3$};

\end{tikzpicture}
\end{center}
\caption{A small absorber for $v$.  The heavy lines are paths.}
\end{figure}
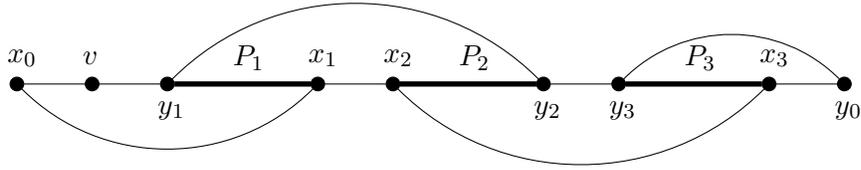

For odd $k$,
$$
x_0 x_1 P_1 y_1 y_2 P_2 x_2 x_3 \dots x_{3k} P_{3k} y_{3k} y_0
$$
and
$$
x_0 v y_1 P_1 x_1 x_2 P_2 y_2 y_3 \dots y_{3k} P_{3k} x_{3k} y_0
$$
are the desired paths.  The size of each absorber is
$$
6k+1 - 3k (6k-2) +1 = 18k^2 + 2
$$
\end{proof}

Observe that given two absorbers $(R,r,u)$ and $(R',r',u')$ for the vertices $v$ and $v'$ respectively, a path from $u$ to $r'$ will generate an absorber capable of absorbing $v$ and $v'$.  Inductively, linking a series of absorbers yields one large absorber.
We now show how to connect absorbers to create a robust subset in the following sense.

\begin{lemma}\label{lemma:robustset}
Let $G$ be a graph with $|V(G)| = s(n) \geq \frac{n}{K}$ and $l = 1000k^2$.  Consider disjoint sets $A, W, X = \{x_1, \dots x_{3r}\}$, $Y = \{y_1, \dots , y_{3r}\}$ and $G$ $n^{\eps}$-expands into $W$.  For $|A|=2r$ and $r \leq 7|W|/90l$, then there exists a subset $W' \subset W$ of size $3r(l-2)-r$ such that given any subset $A'\subset A$ of size $r$ there is a set of $3r$ vertex disjoint $x_i, y_i$ paths of length $l-1$.  Note such paths cover $W' \cup A'$.
\end{lemma}

\begin{proof}
By Lemma \ref{lemma:splitexpand}, we can partition $W$ as $W_1 \sqcup W_2 \sqcup W_3$ with $|W_1|=|W_2|= |W_3| = |W|/3$ and $n^{\eps}$-expands into $W_1$, $W_2$, and $W_3$.  Fix a set $B \subset W_1$ with $|B| = |A|$.  In $W_2$, we can use Lemma \ref{lemma:absorbers} to construct disjoint absorbers $(R_{v,j}, r_{v,j}, u_{v,j})$ for all $v \in A \cup B$ and $1 \leq j \leq 40$, where each absorber $R_{v,j}$ has size $18k^2 + 2$ and can absorb $v$.

The bipartite graph $H$ in Lemma \ref{lemma:smallbipartite} will be our guide as we route paths through the absorbers.  Associate the vertices of the bipartite graph $H$ with $[3r]$ and $A \cup B$ so that the max degree is $40$ and given any subset $A' \subset A$ there exists a perfect matching from $[3r]$ to $A' \cup B$.  For each $v \in A\cup B$, we let $c_v: N_H(v) \rightarrow [40]$ be an injective function.

For a vertex $j \in [3r]$ and absorbers $R_{v,c_v(j)}$, $v \in N_H(j)$. Consider the set
$$
\{r_{v,c_v(j)}, s_{v,c_v(j)}: v \in N(j) \} \cup \{x_j,y_j\}.
$$
Using Lemma \ref{lemma:nonspanning}, we can find paths of length at least $5k$ between pairs of vertices using $W_3$ to create an absorber $(S_j, x_j, y_j)$ of size $l-1$.  As $3lr \leq \frac{7|W_3|}{10}$, this procedure can be done for all indices $[3r]$ simultaneously and disjointly.

Observe that $W' = (\cup_i (S_i \setminus \{x_i, y_i\})) \cup B$ has the desired properties.  For any $A' \subset A$ of size $r$, from the property of the bipartite graph in Lemma \ref{lemma:smallbipartite}, there is a perfect matching between $A' \cup B$ and $[3r]$.  For each $i \in [3r]$, pick the vertex $v \in A'\cup B$ paired with $i$ in the perfect matching and use the $x_i, y_i$-path of length $l-1$ through $S_i \cup \{v\}$.  These paths cover $W' \cup A'$.
\end{proof}

\subsection{Spanning Connecting Lemma}
Now we aggregate the tools of the last two sections to create a spanning version of the connecting lemma.

\begin{lemma}[Spanning Connecting Lemma]
Let $G$ be a graph on $s(n) \geq n/K$ vertices.  Let $t, l$ be integers and $\{(x_i,y_i)\}_{i=1}^{t}$ be a family of pairs of vertices in $V(G)$ such that
\begin{enumerate}[(i)]
\item $1000 k^2 \leq l \leq 2K$
\item $x_i \neq x_j$ and $y_i \neq y_j$ for every distinct $i$ and $j$
\item $t (l-1) = |V(G) \setminus \bigcup_{i=1}^t \{x_i,y_i\}|$
\item $G$ $n^{\eps}$-expands into $V(G) \setminus \bigcup_{i=1}^t \{x_i,y_i\}$
\end{enumerate}

Then there exist $s(n)/t$ (internally) vertex-disjoint directed paths of length $l$ from $x_i$ to $y_i$ for all $i \in [1,t]$ with vertices in $V(G) \setminus \bigcup_{i=1}^t \{x_i,y_i\}$.
\end{lemma}

\begin{proof}
We first show that it suffices to consider $l = 1000k^2$.  Assume the lemma holds for $l_0 = 1000k^2$.  By Lemma \ref{lemma:splitexpand}, we can partition $W$ as $W_1 \sqcup W_2 \sqcup W_3$ with $|W_1| = \frac{7|W|}{10}$, $|W_2| = |W_3| = \frac{3|W|}{20}$.  Now, let $l = l_0 q + r$ where $0 \leq r < l_0$.Take $n/l$ vertices from $W_2$ and label them $z_1, \dots , z_{n/l}$.  Using Lemma \ref{lemma:nonspanning} and $W_1$, we can create disjoint paths of length $\max(l_0/2,r)$ from $x_i$ to $z_i$.  Since $ \max(l_0/3, r)\leq l/2$, we have that the paths cover at most $n/2 \leq 7|W_1|/10$ vertices.  If $l_0/3$ is greater than $r$ then delete $l_0/3 -r$ vertices from the end of these paths starting with $z_i$.  We reuse the variable $z_i$ to denote the new end of such paths for $i \in [n/l]$.

Since between any two sets of size $n^{1-\eps/3}$ there exists an edge, we can greedily find $(n/l)(q-1)$ disjoint edges inside the remaining vertices $W_1 \cup W_2$.  Label these edges as $y_{i,j}x_{i,j+1}$, $i \in [n/l]$ for $1 \leq j <q$.Set $x_{i,1} = z_i$ and $y_{i,q} = y_i$.  Let $W'$ be the vertices in $W$ not in any of these edges or in the paths ending in $z_i$.  Then the pairs $\{(x_{i,j},y_{i,j}): i \in [n/l], j \in [q]\}$, the set $W'$ and the graph $G[W' \cup \{(x_{i,j},y_{i,j}): i \in [n/l], j \in [q]\}]$ satisfy the conditions of the lemma with $l = 10k^2$.

Therefore, we now assume that $l = 1000k^2$.
Fix $s = n^{1- \eps/100} $.  By Lemma \ref{lemma:splitexpand}, we can partition $W$ as $W_1 \sqcup W_2 \sqcup W_3$ with $|W_1| = 5ks$, $|W_2| = 15ks$, $|W_3| = |W|-20ks$ and $G$ $n^\eps$-expands into each.

By Lemma \ref{lemma:robustset}, with $r = 10ks$, we can find a subset $W' \subset W_3$ of size $30ks(l-2) - 10ks$ such that given any subset $Z \subset W_1 \cup W_2$, with $|Z|=r$ there is a set of disjoint $x_i,y_i$ paths $i \in [3r]$ of length $(l-1)$ that cover $W' \cup Z$.

Let $Z_1 = W_3 \setminus W'$ and let $\alpha = l/(10k+2)$ be an integer.  Take distinct vertices $x_{i,j} \in Z_1$ for $3r+1 \leq i \leq n/l$, $2 \leq j \leq \alpha$, and let $x_{i,1} = x_i$ and $x_{i,\alpha+1} = y_i$, for $3r+1 \leq i \leq n/l$.

Let $Z_2 \subset Z_1$ be the remaining unlabeled vertices.  Now we connect as many pairs $x_{i,j}, x_{i,j+1}$ by paths of length $10k+2$ stopping when this is no longer possible or when there are $s$ vertices remaining.  If there are $t >s$ remaining, then take among them $n^{1-\eps/3}$ pairs $(a_i,b_i)$, $i \in [n^{1-\eps/3}]$ so that the vertices $a_i$ and $b_i$ are all distinct.

Let $Z_3 \subset Z_2$ be those vertices not covered by any paths so far. Since
$$
|Z_3| \geq s \geq 1000m
$$
Thus, for any $U \subset V(G)$ with $|U| = n^{1-\eps/3}$,
$$
N(U,Z_3) \geq |Z_3| - 2m \geq (1-1/100) |Z_3|
$$
so the conditions of Lemma \ref{lemma:singlepair} are satisfied and there exists a path of length $10k+2$ from some $a_i$ to $b_i$ contradicting $t > s$.  Thus, the process terminates with exactly $s$ unconnected pairs.

Call these $s$ unconnected pairs $(c_i, d_i)$ with $i \in [s]$.  Let $Z_3 \subset Z_2$ be those vertices not covered by any paths and note that
$$
|Z_3| = (10k+1)s - \frac{|W_1|+|W_2|}{2} = s
$$
Label $Z_3$ as $\{e_i: i \in [s]\}$.
By the $n^{\eps}$ expansion of $G$ into $W_1$, we can find a generalized matching from $(\cup_{i \in [s]} \{c_i, d_i\}) \cup Z_3$ into $W_1$ so that $c_i$ and $d_i$ are each matched to one vertex, $c_i'$ and $d_i'$ respectively, and $e_i$ are matched to two vertices $e_i'$ and $e_i''$. Since
$$
s (5k+1) \leq \frac{7 |W_2|}{10}
$$
We can use Lemma \ref{lemma:nonspanning} and $W_2$ to find paths of length $5k +1$ connecting $c_i'$ to $e_i'$ and $d_i'$ to $e_i''$.  This creates paths of $l$ from all $x_i$ to $y_i$ for $i \geq 3r+1$ and we have used exactly $r$ vertices.  Thus, Lemma \ref{lemma:robustset} allows us to cover the remaining vertices with paths of length $l$ completing the proof.

\end{proof}

\end{document}